\documentclass[a4paper,12pt]{article}

\newenvironment{proof}{\noindent {\bf Proof:}}{\hfill $\Box$}

\newtheorem{theorem}{Theorem}
\newtheorem{lemma}{Lemma}

\newtheorem{corollary}{Corollary}

\newtheorem{assumption}{Assumption}

\usepackage{booktabs}

\usepackage{amsmath}
\usepackage{amssymb}
\usepackage{amsfonts}
\usepackage{graphicx}

\usepackage{color}
\usepackage{ulem}

\title{\bf Convergence rates of moment-sum-of-squares hierarchies for volume approximation of semialgebraic sets}

\begin{document}

\author{Milan Korda$^1$, Didier Henrion$^{2,3,4}$}

\footnotetext[1]{Department of Mechanical Engineering, University of California, Santa Barbara, CA 93106-5070. {\tt milan.korda@engineering.ucsb.edu}.}
\footnotetext[2]{CNRS; LAAS; 7 avenue du colonel Roche, F-31400 Toulouse; France. {\tt henrion@laas.fr}}
\footnotetext[3]{Universit\'e de Toulouse; LAAS; F-31400 Toulouse; France.}
\footnotetext[4]{Faculty of Electrical Engineering, Czech Technical University in Prague,
Technick\'a 2, CZ-16626 Prague, Czech Republic.}

\date{Draft of \today}

\maketitle

\begin{abstract}
Moment-sum-of-squares hierarchies of semidefinite programs can be used to approximate the volume of a given compact basic semialgebraic set $K$. The idea consists of approximating from above the indicator function of $K$ with a sequence of polynomials of increasing degree $d$, so that the integrals of these polynomials generate a convergence sequence of upper bounds on the volume of $K$. We show that the asymptotic rate of this convergence is at least $O(1 / \log \log d)$.
\end{abstract}

\begin{center}\small
\end{center}

\section{Introduction}

The moment-sum-of-squares hierarchy (also known as the Lasserre hierarchy) of semidefinite programs was originally introduced in the context of polynomial optimization, and later on extended in various directions, see e.g. \cite{l09} for an introductory survey and \cite{l10} for a comprehensive treatment.

In our companion note \cite{khj16} we have studied the convergence rate of this hierarchy applied to optimal control. In the present rate, we carry out the convergence analysis of the hierarchy tailored in \cite{hls09} for approximating the volume (and all the moments of the uniform measure) of a given semi-algebraic set.

Let us first recall the main idea of \cite{hls09}. Given a
compact basic semi-algebraic set
\begin{equation}\label{eq:K}
K := \{ x \in \mathbb{R}^n : g_i^K(x) \ge 0, \: i = 1,\ldots,n_{K} \}
\end{equation}
described by polynomials $(g_i^K)_i \subset \mathbb{R}[x]$, we want to approximate numerically its volume
\[
\mathrm{vol}\:K := \int_K dx.
\]
We assume that we are also given a compact basic semi-algebraic set
\begin{equation}\label{eq:K}
X := \{ x \in \mathbb{R}^n : g_i^X(x) \ge 0, i = 1,\ldots,n_{X} \}
\end{equation}
described by polynomials $(g_i^X)_i \subset \mathbb{R}[x]$, such
that $K$ is contained in $X$, which is in turn contained in the unit Euclidean ball:
\[
K \subset X \subset B_n := \{x \in {\mathbb R}^n : \|x\|_2 \leq 1\}
\]
and such that the moments of the uniform (a.k.a. Lebesgue) measure over  $X$ are known analytically or are easy to compute, or equivalently, such that the integral of a given polynomial over $X$ can be evaluated easily.
This is the case for instance if $X$ is a box or a ball, e.g. $X=B_n$. The set $K$, on the contrary, can be a complicated (e.g., non-convex or disconnected) set for which the volume (as well as the higher Lebesgue moments) are hard to compute. In addition we invoke the following technical assumption:
\begin{assumption}\label{ass:Int}
The origin belongs to the interior of $K$, i.e, $0 \in \mathrm{int}\:K$.
\end{assumption}
This assumption can be satisfied whenever the set $K$ has a non-empty interior by translating the set such that the origin belongs to the interior.

For notationally simplicity, we let $g_0^K := g_0^X := 1$. Moreover, since $K$ respectively $X$ are included in the unit ball $B_n$, we assume without loss of generality that one of the $g_i^K$ resp. $g_i^X$, $i>0$, is equal to $1 - \sum_{k=1}^n x_k^2$. Attached to the set $K$ is a specific convex cone called truncated quadratic module defined by
\[
Q_{d}(K)  := \Big\{ \sum_{i=0}^{n_K} g_i^K s_i^K  : s_i^K \in \Sigma[x], \: \mathrm{deg}(g_i^K s_i^K) \leq 2 \lfloor\frac{d}{2}\rfloor \Big\}
\]
where $\Sigma[x] \subset \mathbb{R}[x]$ is the semidefinite representable convex cone of polynomials that can be written as sums of squares of polynomials. Attached to the outer bounding set $X$, we define analogously the truncated quadratic module
\[
Q_{d}(X)  := \Big\{ \sum_{i=0}^{n_X} g_i^X s_i^X  : s_i^X \in \Sigma[x], \: \mathrm{deg}(g_i^X s_i^X) \leq 2 \lfloor\frac{d}{2}\rfloor \Big\}.
\]
Consider then the following hierarchy of convex optimization problems indexed by degree $d$:
\begin{equation}\label{opt:sos}
\begin{array}{rclll}
v_d^\star & := & \inf\limits_{p \in \mathbb{R}[x]_d} &  \displaystyle\int_X p   \\
&& \hspace{0.6cm} \mathrm{s.t.}  & p - 1 \in Q_d(K)\\
&&& p \in Q_d(X)
\end{array}
\end{equation}
where the decision variable belongs to $\mathbb{R}[x]_d$, the finite-dimensional vector space of polynomials of total degree less than or equal to $d$. In problem (\ref{opt:sos}) the objective function is a linear function, the integral of $p$ over $X$, for which we use the simplified notation
\[
\int_X p := \int_X p(x)dx
\]
here and in the remainder of the document. Since we assume that $X$ is chosen such that the moments of the uniform measure are known, the objective function is an explicit linear function of the coefficients of $p$ expressed in any basis of ${\mathbb R}[x]_d$. The truncated quadratic modules $Q_d(K)$ and $Q_d(X)$ are semidefinite representable, so
problem (\ref{opt:sos}) translates to a finite-dimensional semidefinite program. Moreover, it is shown in \cite[Theorem 3.2]{hls09} that the sequence $(v_d^\star)_{d \in \mathbb N}$ converges from above to the volume of $K$, i.e.
\[
v_d^\star \geq v_{d+1}^\star \geq \lim_{d\to \infty} v_d^\star=\mathrm{vol}\:K.
\]
The goal of this note is to analyze the rate of convergence of this sequence.

\section{Convergence rate}

Let us first recall several notions from approximation theory. Given a bounded measurable function $f : X \to \mathbb R$ we define
\[
\begin{array}{rrl}
\omega_f^x(t) := & \displaystyle \sup_{y \in X} & |f(y) - f(x) | \\ 
& \mathrm{s.t.} & \| y -x  \|_2 \le t
\end{array}
\]
as the modulus of continuity of $f$ at a point $x \in X$. In addition, we define
\[
\bar{\omega}_f(t) := \int_X \omega_f^x(t) dx
\]
as the averaged modulus of continuity of $f$ on $X$. Finally we define
\[
\begin{array}{rrl}
e_f(d) := & \displaystyle \inf_{p \in \mathbb{R}[x]_d} & \displaystyle \int_X (p-f) \\
& \mathrm{s.t.} & p \ge f \; \mathrm{on}\; X
\end{array}
\]
as the error of the best approximation to $f$ from above in the $L_1$ norm by polynomials of degree no more than $d$.

\begin{theorem}[\cite{b10}]\label{thm:bhaya}
For any bounded measurable function $f$ and for all $d \in \mathbb N$ it holds
\begin{equation}
e_f(d) \le \bar{c}\:\bar{\omega}_f(1/d)
\end{equation}
for some constant $\bar{c}$ depending only on $f$.
\end{theorem}

Theorem~\ref{thm:bhaya} links the $L_1$ approximation error from above to the averaged modulus of continuity. Now we derive an estimate on the averaged modulus of continuity when $f = I_K$, the indicator function of $K$ equal to 1 on $K$ and 0 outside.

\begin{lemma}\label{lem:tau_est}
For all $t \in [0,1]$ it holds
\[
\bar{\omega}_{I_K}(t) \le c \: t
\]
for some constant $c$ depending only on $K$ and $X$.
\end{lemma}
\begin{proof}
Let $\partial K$ denote the boundary of $K$ and let
\[
\partial K(t) := \{ x \in X : \min_{y \in \partial K} \|x-y\|_2 \le t \}
\]
denote the set of all points whose Euclidean distance to the boundary of $X$ is no more than $t$.
Then we have $\omega_{I_K}^x(t) = 0$ for all $x \not\in \partial K(t)$. Hence
\[
\bar{\omega}_{I_K}(t) = \int_X \omega^x_{I_K}(t) dx = \int_{\partial K(t)} \omega_{I_K}^x(t)dx \le \int_{\partial K(t)} 1\, dx = \mathrm{vol}_:\partial K (t).
\]
Therefore we need to bound the volume of $\partial K(t)$. In order to do so, we decompose
\[
\partial K = \bigcup_{i=1}^{n_\partial} \partial K_i
\]
where each $\partial K_i$ is a smooth manifold of dimension no more than $n-1$ and $n_\partial$ is finite.  Defining
\[
\partial K_i(t) := \{ x \in X : \min_{y \in \partial K_i}\|x-y\|_2 \le t \},
\]
Weyl's tube formula \cite{w39} yields the bound
\[
\mathrm{vol}\:\partial K_i(t) \le c_i \: t
\]
for all $t\in [0,1]$ and for some constant $c_i$ depending on $K$ and $X$ only. Hence we get
\[
\mathrm{vol}\:\partial K(t) \le \sum_{i=1}^{n_\partial} \mathrm{vol}\:\partial K_i(t) \le \sum_{i=1}^{n_\partial}c_i\:t = c\:t 
\]
where $c :=  \sum_{i=1}^{n_\partial}c_i$ depends on $K$ and $X$ only, as desired.
\end{proof}

This leads to the following crucial estimate on the rate of convergence of the best polynomial approximation from above to the indicator function of $K$:
\begin{theorem}\label{thm:approx_main}
For all $d \in \mathbb N$ it holds
\begin{equation}\label{eq:convRate_approx}
e_{I_K}(d) \le \frac{c_1}{d}
\end{equation}
for some constant $c_1$ depending only on $K$ and $X$.
\end{theorem}
\begin{proof}
Follows immediately from Theorem~\ref{thm:bhaya} and Lemma~\ref{lem:tau_est} with $c_1 := \bar{c}\: c$.
\end{proof}

We need the following assumption, which we conjecture is without loss of generality:
\begin{assumption}[Finite Gibbs phenomenon]\label{as:Gibbs}
For $d \in \mathbb N$ there is a sequence
\[
\begin{array}{rrl}
p_d \in & \underset{p \in \mathbb{R}[x]_d}{\mathrm{argmin}} & \displaystyle\int_X (p - I_K) \\
& \mathrm{s.t.} & p \ge I_K \; \mathrm{on}\; X
\end{array}
\]
and a finite constant $c_G$ such that $\max_{x \in X} p_d(x) \le c_G $.
\end{assumption}

In addition to the results from approximation theory, we need the following result which is a consequence of the fundamental result of~\cite{ns07} and of \cite[Corollary~1]{khj16}:

\begin{theorem}\label{thm:nie}
Let $p \in \mathbb{R}[x]$ be strictly positive on a set $S \in \{K,X\}$ and let Assumption~\ref{ass:Int} hold. Then $p \in Q_d(S)$ whenever
\[
d \ge c_2^S \: \mathrm{exp}\Big[\Big ( k(\deg p)(\deg p)^2 n^{\deg p}\frac{\max_{x\in S} p(x)}{\min_{x \in S} p(x)}  \Big)^{c_2^S}\Big]
\]
for some constant $c_2^S$ depending only the algebraic description of $S$ and with $k(d) = 3^{d+1}r^d$, where
\begin{equation}\label{eq:rdef}
r := \frac{1}{\sup\{s > 0 \mid [-s,s]^{n} \subset K \}},
\end{equation}
which is the reciprocal value of the length of the side of the largest box centered around the origin included in $K$\footnote{By assumption $K \subset X$ and hence it is sufficient to take $K$ in the definition of $r$ in~(\ref{eq:rdef}) rather than $S$.}.
\end{theorem}

Now we are in position to prove our main result:
\begin{theorem}
For all $\varepsilon>0$
it holds $v_d^\star - \mathrm{vol}\:K < \varepsilon $ whenever
\begin{equation}\label{eq:bound}
d  \ge  c_2 \: \mathrm{exp}\Big[\Big(  \frac{3 c_3(\epsilon)^2  (3rn)^{c_3(\epsilon)}(2c_G\:\mathrm{vol}\:B_n+\varepsilon)}{\varepsilon}    \Big)^{c_2}      \Big]  =
O\Big(\mathrm{exp}\Big[\frac{1}{\varepsilon^{3c_2}} (3rn)^{\frac{2c_1}{\varepsilon}}\Big]\Big)
\end{equation}
where $c_3(\epsilon) := \lceil 2 c_1/\varepsilon\rceil$, $r$ is defined in~(\ref{eq:rdef}) and the constants $c_1$ and $c_2$ depend only on $X$ and $K$.
\end{theorem}
\begin{proof}
Let $\varepsilon > 0$ be fixed and decompose $\varepsilon = \varepsilon_1 + \varepsilon_2$ with $\varepsilon_1 >0$ and $\varepsilon_2 > 0$. Let $p \in \mathbb{R}[x]_d$ be any polynomial feasible in~(\ref{opt:sos}) for some $d \ge 1$. Then we have
\begin{align*}
v_d^\star - \mathrm{vol}\:K = v_d^\star - \int_X I_K \le \int_X(p - I_K) \le \int_X (\tilde{p}_{\tilde{d}} - I_K) + \int_X |p - \tilde{p}_{\tilde{d}}|
\end{align*}
where $\tilde{p}_{\tilde{d}}$ is an optimal polynomial approximation to $I_K$ from above of degree $\tilde{d}$ satisfying Assumption~\ref{as:Gibbs}, i.e.,  $\int_X (\tilde{p}_{\tilde{d}} - I_K) \le c_1 /\tilde{d}$ and $\max_{x\in X} \| \tilde{p}_{\tilde{d}}(x)\| \le c_G$. Selecting $\tilde{d} = \lceil c_1 / \varepsilon_1 \rceil$ we obtain 
\[
v_d^\star - \mathrm{vol}\:K \le \varepsilon_1 + \int_X |f - \tilde{p}_{\tilde{d}}|
\]
Hence it suffices to find a $d\ge 0$ and a polynomial $p$ feasible in~(\ref{opt:sos}) for this $d$ such that $\int_X |p - \tilde{p}_{\tilde{d}}| \le \varepsilon_2$. Let $p := \tilde{p}_{\tilde{d}} + \varepsilon_2 / \mathrm{vol}\,B_n$, then
\[
\int_X |p - \tilde{p}_{\tilde{d}}| = \int_X \frac{\varepsilon_2}{\mathrm{vol}\:B_n} \le \frac{\varepsilon_2}{\mathrm{vol}\:B_n}\int_{B_n} =   \varepsilon_2.
\]
In addition, since $\tilde{p}_{\tilde{d}} \ge I_K$ on $X$, we have $p \ge \varepsilon_2/\mathrm{vol}\:B_n > 0$ on $X$ and $p - 1 \ge \varepsilon_2/\mathrm{vol}\:B_n > 0$ on $K$ and hence $p$ is feasible in~(\ref{opt:sos}) for some $d\ge 0$~\cite{p93}. To bound the degree $d$ we invoke Theorem~\ref{thm:nie} on the two constraints of problem~(\ref{opt:sos}). Set $c_2 := \max(c_2^X,c_2^K)$, where the constants $c_2^X$ and $c_2^K$ are the constants from~Theorem~\ref{thm:nie} (note that these constants depend only on the algebraic description of the sets, not on the polynomial $p$). Since $\max_{x\in K}|p(x)-1| \le \max_{x \in X} |p(x)| \le c_G+\varepsilon_2/ \mathrm{vol}\:B_n$ and $d_p(\epsilon_1) :=\mathrm{deg}\:p = \tilde{d} = \lceil c_1/\varepsilon_1\rceil$, Theorem~\ref{thm:nie} applied to the two constraints of the problem~(\ref{opt:sos}) yields
\[
d \ge c_2 \: \mathrm{exp}\Big[\Big(  \frac{ k(d_p(\epsilon_1)) d_p(\epsilon_1)^2 n^{d_p(\epsilon_1)}   (c_G+\varepsilon_2 / \mathrm{vol}\:B_n)}{\varepsilon_2 / \mathrm{vol}\:B_n}    \Big)^{c_2}      \Big]
\]
with $k(d) = 3^{d+1}r^d$ and $r$ is defined in~(\ref{eq:rdef}). Setting $\varepsilon_1 = \varepsilon_2 = \varepsilon /2$ and $c_3(\epsilon) := \lceil 2 c_1/\varepsilon\rceil$, we obtain~(\ref{eq:bound}).

\end{proof}

\begin{corollary}
It holds $v_d^\star - \mathrm{vol}\:K = O\Big(1 / \log \log\, d\Big)$.
\end{corollary}
\begin{proof}
Follows by inverting the asymptotic expression $O\big(\mathrm{exp}[\frac{1}{\varepsilon^{3c_2}} (3rn)^{\frac{2c_1}{\varepsilon}}]\big)$ using the fact that $(3rn)^{\frac{4c_1}{\varepsilon}} \ge \frac{1}{\varepsilon^{3c_2}} (3rn)^{\frac{2c_1}{\varepsilon}}$ for small $\varepsilon$.
\end{proof}

\section{Concluding remarks}

Our doubly logarithmic asymptotic convergence rate is likely to be pessimistic, even though we suspect the actual convergence rate to be sublinear in $d$, see for example \cite[Fig. 4.5]{hls09} which reports on the values of the moment-sums-of-squares hierarchy for an elementary univariate interval length estimation problem. In this case we could compute numerically the values for degrees up to 100 thanks to the use of Chebyshev polynomials. Beyond univariate polynomials, it is however difficult to study experimentally the convergence rate since the number of variables in the semidefinite program grows polynomially in the degree $d$, but with an exponent equal to the dimension of the set $K$.

\end{document}